\documentclass[11pt]{article}

\usepackage{fullpage}
\usepackage{graphicx,hyperref,cite,amsmath,amsthm,mathptm}

\newtheorem{theorem}{Theorem}
\newtheorem{lemma}[theorem]{Lemma}

\newtheorem{conjecture}[theorem]{Conjecture}

\newcommand{\A}{\mathcal{A}}
\renewcommand{\O}{\Omega}

\title{Antimatroids and Balanced Pairs}
\author{David Eppstein\\
\small Computer Science Department, University of California, Irvine}
\date{ }

\begin{document}
\maketitle

\begin{abstract}
We generalize the $\frac13$~--~$\frac23$ conjecture from partially ordered sets to antimatroids: we conjecture that any antimatroid has a pair of elements $x,y$ such that $x$ has probability between $\frac13$~and~$\frac23$ of appearing earlier than $y$ in a uniformly random basic word of the antimatroid. We prove the conjecture for antimatroids of convex dimension two (the antimatroid-theoretic analogue of partial orders of width two), for antimatroids of height two, for antimatroids with an independent element, and for the perfect elimination antimatroids and node search antimatroids of several classes of graphs. A computer search shows that the conjecture is true for all antimatroids with at most six elements.
\end{abstract}

\section{Introduction}

The linear extensions of a finite partial order model the sequences that may be formed by
adding one element at a time to the end of the sequence, with the constraint that an element may only be added once all its predecessors have already been added. This constraint is \emph{monotonic}: once an element becomes available to be added, it remains available until it actually is added. However, not every system of monotonic constraints may be modeled by a partial order. Sets of orderings formed by adding one element at a time to the end of a partial sequence, with monotonic constraints that may depend arbitrarily on the previously added elements of the sequence, may instead be modeled in full generality by mathematical structures known as \emph{antimatroids}.

Because antimatroids generalize finite partial orders, they can model the ordering of objects in  applications for which partial orders are inadequate, including single-processor scheduling~\cite{BoyFai-DAM-90}, event ordering in discrete event simulations~\cite{GlaYao-94}, and the sequencing of lessons in automated learning environments~\cite{DoiFal-LS-11,Epp-LS-13}. These applications, in turn, make it of interest to extend to antimatroids the mathematical problems and results that have been studied for partial orders and their linear extensions. In this paper we extend to antimatroids the well-known $\frac13$~--~$\frac23$ conjecture for partially ordered sets, stating that any non-total partial order has two elements that are nearly equally likely to appear in either order in a uniformly random linear extension. We prove that the conjecture is true  for several special classes of antimatroids, and we show by a computer search that it is true for all antimatroids that have at most six elements.

\section{Preliminaries}
\subsection{Antimatroids}
We recall some standard definitions from antimatroid theory. For a more detailed overview of antimatroids and their theory, see Korte, Lov\'asz and Schrader~\cite{KorLovSch-Greedoids-91}.

An \emph{antimatroid} is a finite family $\A$ of finite sets that is \emph{accessible}---for all nonempty $S\in\A$, there exists $x\in S$ such that $S\setminus\{x\}\in\A$---and closed under unions. We refer to the members of $\bigcup\A$ as the \emph{elements} of $\A$. A \emph{path} in an antimatroid $\A$ is a set $P\in\A$ with the property that there is only one element $x$ of $P$ (the \emph{endpoint} of the path) for which $P\setminus\{x\}\in\A$; every set $S$ in $\A$ may be represented as a union of paths. The \emph{path poset} of $\A$ is the set inclusion ordering on the paths of $\A$.
A \emph{basic word} of $\A$ is a permutation $\sigma$ of its elements with the property that, for each prefix of $\sigma$, the set of elements in the prefix is a member of $\A$.
A \emph{chain antimatroid} is an antimatroid with a single basic word.
The \emph{join} of antimatroids  $\A$ and $\mathcal{B}$ is the family $\{A\cup B\mid A\in\A\wedge B\in\mathcal{B}\}$. Joins are commutative and associative so we may refer without ambiguity to the join of any finite set of antimatroids. Any antimatroid $\A$ may be represented as a join of chain antimatroids, for instance by choosing one chain for each basic word of $\A$; the \emph{convex dimension} of $\A$ is the minimum number of chain antimatroids whose join is $\A$, and equals the width of the path poset of $\A$~\cite{EdeSak-Order-88}.

We model a finite \emph{partially ordered set} as a transitive, reflexive binary relation $\le$ on a finite set of elements.
The family of lower sets of any finite partially ordered set $P$ forms an antimatroid, the \emph{poset antimatroid} of $P$. The poset antimatroid has the linear extensions of $P$ as its basic words, the principal ideals of $P$ as its paths, $P$ itself as its path poset, and the width of $P$ as its convex dimension.

An alternative terminology for the same class of combinatorial objects as antimatroids is given by the theory of convex geometries~\cite{EdeJam-GD-85}, finite families of finite sets that are closed under intersections and that satisfy the property that, whenever a set $A$ belongs to the family but is not the union of the family, there exists $x\notin A$ such that $A\cup\{x\}$ belongs to the family. The sets of a convex geometry are exactly the complements of the sets in an antimatroid. In convex geometry terminology, antimatroid paths correspond to the complements of copoints, endpoints of paths correspond to points of attachment, basic words correspond to compatible orders, and chain antimatroids correspond to monotone alignments. Although antimatroids and convex geometries are mathematically equivalent, we prefer the antimatroid terminology because of the applications of antimatroids in computer-aided instruction, where (under the name ``learning spaces'') they represent the sets of concepts that a human learner is capable of mastering~\cite{DoiFal-LS-11}. In this application, accessibility and closure under unions are well-motivated pedagogically: it should be possible to learn a set of concepts one at a time, and learning one set of concepts should not prevent learning a different set. However, closure under intersections is not generally the case in this application. The sets of orderings given by the basic words of an antimatroid have a natural meaning in this context, as the sequences in which a student may progress through the concepts, and subsets of these orderings may also be used in this setting as a data structure that speeds the computerized assessment of a student's state of knowledge~\cite{Epp-LS-13}. From the point of view of this application, it is important to understand the mathematics of these sets of orderings.

\subsection{The $\mathbf{\frac13}$~--~$\mathbf{\frac23}$ conjecture for partial orders}

For any partially ordered set $P$ containing elements $x$ and $y$, let $\Pr[x<y]$ be the probability that $x$ occurs earlier than $y$ in a linear extension of $P$ selected uniformly at random among all possible linear extensions.
Define the \emph{balance} of $P$ to be
$$\delta(P) = \max_{x,y\in P}\min\left\{ \Pr[x<y],\Pr[y<x]\right\} .$$
That is, we find the pair whose probability of appearing in either order is as close as possible to  $\frac12$ and define the balance of the partial order to be the smaller of the two probabilities for that pair.

We say that $x,y$ is a \emph{balanced pair} if $\frac13\le\Pr[x<y]\le\frac23$.
The well-known $\frac13$~--~$\frac23$ conjecture~\cite{Kis-MZ-68,Sak-Order-85,Bri-DM-99} states that every finite partial order that is not total has a balanced pair, or equivalently $\delta(P)\ge \frac13$. It is known that there exists a constant $c$ such that every finite partial order has $\delta(P)\ge c$~\cite{BriFelTro-Order-95,KahSak-Order-84}, but it is not known that $c=1/3$. It is also known that $\delta(P)\ge\frac13$ whenever $P$ is a semiorder~\cite{Bri-Order-89}, a partial order of height two~\cite{TroGehFis-Order-92}, a partial orders with at most eleven elements~\cite{Pec-Order-06}, or a partial order of width two~\cite{Lin-SJoC-84}. Width-two partial orders are conjectured to be the most difficult case of the $\frac13$~--~$\frac23$ conjecture, in the sense that the minimum balance of width-$w$ partial orders is conjectured to tend to $1/2$ in the limit as $w$ goes to infinity~\cite{KahSak-Order-84}, and the smallest balance of any known partial order with width${}>2$ is $14/39$~\cite{Sak-Order-85,TroGehFis-Order-92}.

\subsection{The $\mathbf{\frac13}$~--~$\mathbf{\frac23}$ conjecture for antimatroids}

By analogy to the definitions in the previous section, for any pair $x$ and $y$ of elements of an antimatroid $\A$, let
$\Pr[x<y]$ be the probability that $x$ occurs earlier than $y$ in a basic word of $\A$ selected uniformly at random among all possible basic words. Again, we define the balance of the antimatroid to be the probability that is closest to $1/2$ among these pairwise ordering probabilities,
$$\delta(\A) = \max_{x,y\in\A}\min\left\{ \Pr[x<y],\Pr[y<x]\right\}.$$

Based on the results in the remainder of this paper, we make the following conjecture:

\begin{conjecture}
\label{conj}
For every antimatroid $\A$ with more than one basic word, $\delta({\A})\ge\frac13$.
\end{conjecture}

Equivalently, if we define a balanced pair in an antimatroid to be a pair $x,y$ of elements for which $\min\left\{ \Pr[x<y],\Pr[y<x]\right\}\ge\frac13$, then the conjecture states that every antimatroid with more than one basic word contains a balanced pair. We define an antimatroid to be \emph{balanced} if it either contains a balanced pair or is a chain antimatroid, and \emph{unbalanced} otherwise; the conjecture states that all antimatroids are balanced.

\section{Sets of orderings}

Before giving our results on antimatroids, we provide some results on more general sets of orderings. Given a finite set $S$ of $n$ elements, we define a \emph{set of orderings} $\O$ to be a nonempty subset of the $n!$ total orderings of $S$. Thus, there are $2^{n!}-1$ possible sets of orderings. We refer to the members of $S$ as \emph{elements} of $\O$, and the members of $\O$ itself as the \emph{orderings} of $\O$.

We remark that sets of orderings, in general, do not obey the $\frac13$--$\frac23$ conjecture. For instance, if $n=5$ and $S=\{0,1,2,3,4\}$, then the set of orderings
$$\O=\{01234, 01243, 01324, 23014\}$$
(defined from a family of sets that is accessible but not closed under unions)
does not have a balanced pair: for every pair $(x,y)$ with $x<y$, $x$ appears earlier than $y$ in three out of four of these orderings. More generally the balance of a set of orderings may be arbitrarily small: the set of orderings on $n$ items consisting of a single permutation $\sigma$ of the items together with the $n-1$ permutations generated by transposing two adjacent elements of $\sigma$ has balance~$1/n$.

Although sets of orderings may not be balanced, certain substructures within them, when present, ensure that the set of orderings has a balanced pair. In the following three subsections we describe three such substructures, \emph{independent elements},  \emph{double ladders}, and \emph{many initial or final elements}. Later in this paper we will use these structures to show that certain specific types of antimatroid are guaranteed to be balanced.

\subsection{Independent elements}

As we have already done for partial orders and antimatroids, define $\Pr[x<y]$ (for any pair $x$ and $y$ of elements of a set $\O$ of orderings) to be the probability that $x$ occurs earlier than $y$ in a ordering of $\O$ selected uniformly at random among all of its ordering.
Define a binary relation $\prec$ on the elements of $\O$, according to which
$x\prec y$ if and only if $\Pr[x<y]>\frac23$. Thus, $\prec$ is antisymmetric, and two elements are related by $\prec$ if and only if they do not form a balanced pair.

\begin{lemma}
\label{lem:prec-is-total}
If $\O$ has no balanced pair then $\prec$ is the ordering relation of an (irreflexive) total order.
\end{lemma}

\begin{proof}
If $\O$ is unbalanced, $\prec$ defines a \emph{tournament}, a binary relation in which every two elements are related in exactly one way. Every tournament $R$ on a finite set of elements either defines a total order (that is, is a \emph{transitive tournament}), or it has three elements $x_0$, $x_1$, and $x_2$ with $x_0 R x_1$, $x_1 R  x_2$, and $x_2 R x_0$~\cite{HarMos-AMM-66}. However, such a cyclic triple is not possible for~$\prec$. For, in any ordering of $\Omega$ in which $x_2$ is earlier than $x_0$, it must also be the case that either $x_1$ precedes $x_0$ or $x_2$ precedes $x_1$. Therefore, $\Pr[x_2<x_0]\le\Pr[x_1<x_0]+\Pr[x_2<x_1]$. But if $x_0\prec x_1\prec x_2\prec x_0$ then $\Pr[x_2<x_0]>\frac23$, while $\Pr[x_1<x_0]<\frac13$ and $\Pr[x_2<x_1]<\frac13$, contradicting this inequality.
\end{proof}

We define an \emph{independent element} of a set of orderings $\O$ to be an element $x$ such that, in every ordering of $\O$, $x$ may be swapped with either of the elements that are adjacent to it in the ordering, producing another ordering that belongs to $\O$. That is, if $x$ is an independent element, it may occur in any position in any ordering of $\O$, independently of how the other elements are ordered. For instance, in a partially ordered set, an independent element is an element that does not have any order relations with other elements (Figure~\ref{fig:independent-element}). In an antimatroid, an independent element is an element $x$ such that $\{x\}$ is a path and such that $x$ is not a member of any other paths.

\begin{figure}[t]
\centering\includegraphics[scale=0.8]{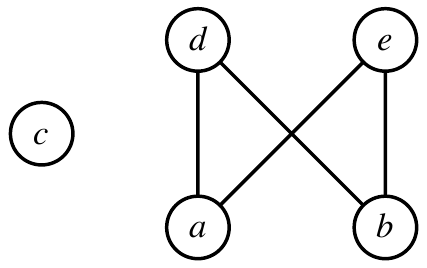}
\caption{The Hasse diagram of a partially ordered set with an independent element.}
\label{fig:independent-element}
\end{figure}

\begin{theorem}
\label{thm:independent-element}
Let $\O$ be a set of orderings containing an independent element. Then $\O$ is balanced.
\end{theorem}

\begin{proof}
Let $\O$ be the given set of orderings, with an independent element $x$, and let $\O_x$ be the set of orderings on one fewer element, formed by removing $x$ from the orderings of $\O$. If $\O_x$ has a balanced pair $(y,z)$, then $\Pr[y<z]$ is the same in $\O$ as it is in $\O_x$, and so $\O$ is also balanced. Otherwise, by Lemma~\ref{lem:prec-is-total}, $\prec$ defines a total order on $\O_x$, and we may number the elements of $\O_x$ as $y_0\prec y_1\prec\cdots\prec y_{n-2}$.

Let $p(y_i)$ be a random variable indicating the position of $y_i$ in a uniformly random ordering of $\O_x$ (where the first element of the ordering has position number one, etc), and let $\bar p(y_i)$ be the expected value of $p(y_i)$. A uniformly random ordering of $\O$ may be obtained by selecting a uniformly random ordering of $\O_x$ and inserting $x$ into one of $n$ positions chosen uniformly at random; $p(y_i)$ of these $n$ positions precede the position of $y_i$. Therefore, $\Pr[x<y_i]=\bar p(y_i)/n$.

We may lower bound $\bar p(y_i)$ by the expected number of elements that precede $y_i$ in the total order of $\prec$ and that also precede $y_i$ in a given randomly chosen ordering of $\O_x$. Each $y_j$ for $j<i$ contributes more than $2/3$ of a unit to this expectation (because its probability of preceding $y_i$ is more than $2/3$), so $\bar p(y_i)\ge 1 + \frac23i$. Similarly, we may upper bound $\bar p(y_i)$ by subtracting  the expected number of elements that follow $y_i$ in the total order of $\prec$ and that also follow $y_i$ in a random ordering of $\O_x$ from the maximum position that $y_i$ may take: $\bar p(y_i)\le n-1-\frac23(n-2-i)$. Choosing $i = \lfloor n/2\rfloor -1$, these lower and upper bounds simplify to $n/3 \le \bar p(y_i) \le (2n-1)/3$, and therefore $\frac13\le \Pr[x<y_i]<\frac23$. Thus, for this choice of $i$, $x$ and $y_i$ form a balanced pair.
\end{proof}

\subsection{Ladders}

Let $\O$ be a set of orderings, $x$ be an element of  $\O$, and $Y=(y_0,y_1,\dots,y_{k-1})$ be a sequence of a subset of the elements of $\O$.  We say that $(x,Y)$ is a \emph{ladder} in $\O$ if it meets the following two conditions:
\begin{enumerate}
\item In all orderings of $\O$, the elements of $Y$ that occur earlier than $x$ necessarily occur in the order given by the sequence $Y$.
\item If $w$ is an ordering of $\O$, some of the elements of $Y$ occur earlier than $x$ in $w$, and $i$ is the largest index of an element of $Y$ that occurs earlier than $x$ in $w$, then the ordering formed by swapping $x$ and $y_i$ also belongs to $\O$.
\end{enumerate}

We define a \emph{double ladder} to be a pair of sequences $X=(x_0,\dots)$ and $Y=(y_0,\dots)$ for which both $(x_0,Y)$ and $(y_0,X)$ are ladders. For instance, if $x$ and $y$ are an indistinguishable pair of elements (that is, swapping $x$ for $y$ in any ordering of $\O$ always produces another ordering of $\O$), then the single-element sequences $X=(x)$ and $Y=(y)$ form a double ladder. As a less-trivial example, if $P$ is a series-parallel partial order, and $X$ and $Y$ are two totally ordered subsets of $P$ that are combined by a parallel composition, then $(X,Y)$ forms a double ladder. As we will show, when this structure exists in a set of orderings, it ensures the existence of a balanced pair.

If $(x,Y)$ is a ladder, we define the \emph{rung} of a ordering $w$ to be the largest index of an element of $Y$ that occurs earlier than $x$ in $w$, or $-1$ if no such element exists, and we let $W_i$ denote the set of orderings of $\O$ whose rung is~$i$. Obviously, the sets $W_i$ form a partition of the orderings of $\O$.

\begin{lemma}
\label{lem:ladder-monotonicity}
If $(x,Y)$ is a ladder for $\O$, the sets $W_i$ are defined as above, and $i$ is an integer in the range $0\le i<k$, then $|W_i|\le |W_{i-1}|$.
\end{lemma}

\begin{proof}
The operation of swapping $x$ and $y_i$ in an ordering of $\O$ forms an injective but not necessarily surjective map from $W_i$ to $W_{i-1}$.
\end{proof}

\begin{lemma}
\label{lem:delta-ladder}
If $(x,Y)$ is a ladder for $\O$, the sets $W_i$ are defined as above, and $i$ is an integer in the range $0\le i<k$, then
$$\Pr[y_i<x]=\frac{\sum_{j=i}^{k-1} |W_j|}{\sum_{j=-1}^{k-1} |W_j|}.$$
\end{lemma}

\begin{proof}
It follows from the definition of a ladder that the orderings in which $y_i$ occurs earlier than $x$ are counted by the numerator of the right hand side of the equation, and all orderings of $\O$ are counted by the denominator.
\end{proof}

\begin{lemma}
\label{lem:ladder-top}
If $(x,Y)$ is a ladder for $\O$, the sets $W_i$ are defined as above, and $k$ is the length of $Y$, then
$\Pr[y_{k-1}<x]\le \frac12$.
\end{lemma}

\begin{proof}
In the expression of Lemma~\ref{lem:delta-ladder} for $\Pr[y_{k-1}<x]$, the sum in the numerator includes only the single term $|W_{k-1}|$, whereas the sum in the denominator includes the terms $|W_{k-1}|$ and $|W_{k-2}|$ as well as possibly some others. By Lemma~\ref{lem:ladder-monotonicity}, these two terms alone are already enough to make the denominator at least twice the numerator.
\end{proof}

\begin{lemma}
\label{lem:small-ladder-is-balanced}
Let $\O$ be a set of orderings that has a ladder $(x,Y)$, and in addition suppose that $\Pr[y_0<x]\ge \frac12$. Then there exists an integer $i$ in the range $0\le i<k$ such that $\frac13\le\Pr[y_i<x]\le\frac23$.
\end{lemma}

\begin{proof}
If $\Pr[y_0<x]\le\frac23$, then we may take $i=0$. Otherwise, by Lemma~\ref{lem:ladder-monotonicity}, each set $W_i$ contains at most $1/3$ of the total number of orderings of~$\O$.
Let $i$ be the largest index for which $\Pr[y_i<x]\ge\frac13$. Then if $i={k-1}$, $\Pr[y_i<x]\le\frac12$ by Lemma~\ref{lem:ladder-top}; otherwise, $\Pr[y_i<x]\le\Pr[y_{i+1}<x]+\frac13\le\frac23$. Thus, in either case, $i$ satisfies the conditions of the lemma.
\end{proof}

\begin{theorem}
\label{thm:double-ladder-is-balanced}
Let $\O$ be a set of orderings that contains a double ladder $(X,Y)$. Then $\O$ is balanced.
\end{theorem}

\begin{proof}
If $\Pr[y_0<x_0]\ge\frac12$ then the ladder $(x_0,Y)$ satisfies the conditions of Lemma~\ref{lem:small-ladder-is-balanced}, and otherwise the ladder $(y_0,X)$ satisfies the conditions of the lemma. In either case the lemma ensures the existence of a balanced pair.
\end{proof}

\subsection{Initial and final elements}

Define an \emph{initial element} of a set of orderings $\O$ to be an element $x$ with the property that, in any ordering of $\O$ in which $x$ is not already the first element, $x$ may be swapped with its immediate predecessor to form another ordering of $\O$. Symmetrically, define a \emph{final element} of $\O$ to be an element $y$ with the property that, in any ordering of $\O$ in which $y$ is not already the last element, $y$ may be swapped with its immediate successor to form another ordering of~$\O$. For instance, in an antimatroid $\A$, $x$ is an initial element if $\{x\}$ is a path of $\A$, and $y$ is a final element if it is the endpoint of every path containing it. As a special case, in any set of orderings, the elements that are both initial and final are exactly the independent elements. As we will show, sets of orderings with many initial elements, or symmetrically with many final elements, are necessarily balanced.

We define the \emph{rank} of an element $x$ in an ordering $w$ to be the position of $x$ in $w$, numbering the first symbol of $w$ as position one, etc. We define the \emph{expected rank} $\bar r(x)$ (for a particular set of orderings) to be the expected value of the rank of $x$ in an ordering drawn uniformly at random from $\O$. Given two distinct elements $x$ and $y$, and two integers $i$ and $j$, we define $f_{x,y}(i,j)$ to be the probability that, in a uniformly random ordering, $x$ has rank $i$ and $y$ has rank $j$.
Obviously, $f_{x,y}(i,i)=0$ for all $i$ and all $x\ne y$, because only one of the two elements may have rank~$i$ in any particular ordering.

\begin{lemma}
\label{lem:diagonal-equality}
For any two initial elements $x$ and $y$, and any $i$, $f_{x,y}(i,i+1)=f_{x,y}(i+1,i)$.
\end{lemma}

\begin{proof}
In the event that $x$ has rank $i$ and $y$ has rank $i+1$, these two initial elements are adjacent in the ordering, and can be swapped to produce an equally likely event that $x$ has rank $i+1$ and $y$ has rank $i$.
\end{proof}

\begin{lemma}
\label{lem:monotonicity}
For any two initial elements $x$ and $y$, any $i>1$, and any $j\ne i+1$, $f_{x,y}(i,j-1)\ge f_{x,y}(i,j)$.
\end{lemma}

\begin{proof}
In the event that $x$ has rank $i$ and $y$ has rank $j$, $y$ may be swapped for its predecessor, producing another event that is therefore at least as likely in which $y$ has rank $j-1$ and the rank of $x$ is unchanged.
\end{proof}

\begin{lemma}
\label{lem:rank-ub}
Let $x$ and $y$ be two distinct initial elements of a set of orderings $\O$ with $n$ elements, and let $x\prec y$.
Then $\bar r(x)<\frac{n+1}{3}$.
\end{lemma}

\begin{proof}
For any integer $i$, define
$$g(i) = \sum_{1\le j\le i} f_{x,y}(i+1,j)$$
and
$$h(i) = \sum_{1\le j\le i} f_{x,y}(i+1,j) + \sum_{i<j\le n} f_{x,y}(i,j).$$
Then $g(i)$ is a sum of $i$ terms, while $h(i)$ is a sum of $n$ terms. By Lemmas~\ref{lem:diagonal-equality} and~\ref{lem:monotonicity}, each term in $h(i)$ that is not already present in $g(i)$ is at most equal to the smallest term in $g(i)$, and therefore at most equal to the average value in $g(i)$. Expressing this as an inequality, we have
$h(i)\le n g(i)/i$.

The probability $\Pr[x>y]$ is $\sum g(i)<\frac13$ by the assumption that $x\prec y$. We can also express the expected rank of $x$, as
$$\bar r(x)=\sum_{i,j} if_{x,y}(i,j) = \sum_i g(i) + i h(i);$$
the two summations in this formula are equal, as may be seen by comparing the contributions of each individual term $f_{x,y}(i,j)$ to each of them.

Combining this expression with the two previously observed inequalities, we have
$$\bar r(x)=\sum_i g(i) + i h(i)\le \sum (n+1)g(i)=(n+1)\sum g(i)<\frac{n+1}{3}.$$
\end{proof}

\begin{lemma}
\label{lem:rank-lb}
Let $x$ be an initial element of a set of orderings $\O$ that is preceded in relation $\prec$ by $i\ge 2$ other initial elements, and suppose that relation $\prec$ is total in $\O$. Then $\bar r(x)> (2i+5)/3$.
\end{lemma}

\begin{proof}
The expected rank of $x$ is one plus the expected number of elements of $\O$ that precede $x$ in a uniformly random ordering. If $S$ is the set of initial elements that precede $x$ in relation $\prec$, then the expected contribution of $S$ to the number of elements that precede $x$ in a random ordering is
$$\sum_{y\in S}\Pr[y<x] > \frac{2i}{3}.$$
Adding the one unit for the fact that the rank is one even when there are zero elements preceding $x$, we obtain a total of $(2i+3)/3$.

To obtain the remaining 2/3 of a unit in the expected rank given in the statement of the lemma, let $y$ be the immediate predecessor of $x$ in $\prec$, and let $z$ be the immediate predecessor of $y$ in $\prec$ (this is where we use the assumption that $i\ge 2$). Let $E$ be the event that $y$ precedes $x$ in an ordering and that at least one non-initial element occurs between $y$ and $x$ in the word. Similarly, let $F$ be the event that $z$ precedes $y$ in an ordering and that at least one non-initial element occurs between $z$ and $y$ in the word. In any word for which $E$ is not true, either $x$ precedes $y$ or $x$ may be swapped for $y$ to produce another word in which $x$ precedes $y$; by the assumption that $y\prec x$, either of these events must happen with probability less than $1/3$, so $\Pr[E]>\frac13$, and by a similar argument $\Pr[F]>\frac13$. But whenever $E$ or $F$ happens, there is a non-initial element earlier than $x$, and the non-initial elements that cause $E$ to happen are disjoint from the non-initial elements that cause $F$ to happen. Therefore, the expected number of non-initial elements earlier than $x$ is at least $\Pr[E]+\Pr[F]>\frac23$. Adding this to the already-obtained total of $(2i+3)/3$ for the initial rank of one and the expected number of initial elements earlier than $x$ produces the result.
\end{proof}

\begin{theorem}
\label{thm:many-initial}
Let $\O$ be a set of orderings with at least seven elements, in which either at least half of the elements are initial or at least half of the elements are final. Then $\O$ is balanced.
\end{theorem}

\begin{proof}
Assume without loss of generality that at least half of the elements are initial; otherwise, reversing all of the orderings in $\O$ swaps the roles of the initial and final elements without changing whether $\O$ is balanced.

Assume for a contradiction that $\O$ is unbalanced, let $x$ be the second-highest initial element in the total ordering given by $\prec$ according to Lemma~\ref{lem:prec-is-total}, and let $n$ be the number of elements in $\O$. Then by Lemma~\ref{lem:rank-ub}, $\bar r(x)<(n+1)/3$, while by Lemma~\ref{lem:rank-lb}, $\bar r(x)>(2(n/2-2)+5)/3=(n+1)/3$. The contradiction shows that $\O$ must be balanced.
\end{proof}

\section{Antimatroids in general}

\begin{figure}[t]
\centering\includegraphics[width=\textwidth]{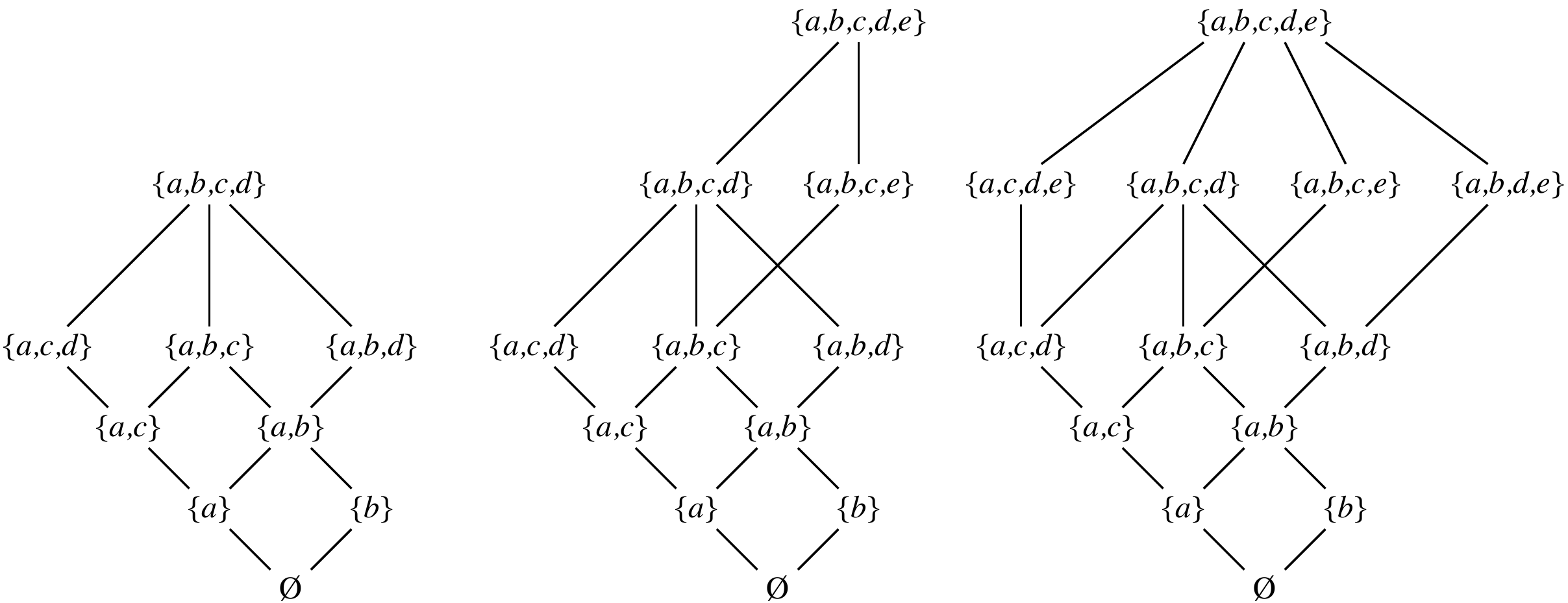}
\caption{Three antimatroids with balance exactly~$1/3$.}
\label{fig:extremes}
\end{figure}

\subsection{Computational enumeration}
In order to test Conjecture~\ref{conj} without making any additional assumptions on the type of antimatroid being tested, we performed a computer search for small antimatroids, aided by the following lemma:

\begin{lemma}
\label{lem:rev-search-parent}
Let $\A$ be an antimatroid that is not the power set of its elements. Then there is a set $S\notin\A$ such that $\A\cup\{S\}$ is also an antimatroid.
\end{lemma}

\begin{proof}
We apply induction on the number of elements of $\A$.
Let $X_0$ be the set of elements of $\A$, and (for each $i>0$) let $X_i$ be a set in $\A$ formed by removing an element from $X_{i-1}$. If $\A$ is not the power set of its elements, let $j$ be the smallest number such that the interval between $X_j$ and $X_0$ has fewer than $2^j$ sets in it, and let $x$ be the element removed from $X_{j-1}$ to form $X_j$.

If $X_0\setminus\{x\}$ belongs to $\A$, let $\A'$ be the antimatroid $\{T\mid (T\cap X_j=\emptyset)\wedge (T\cup X_j\in\A)\wedge(x\notin T)\}$.   By induction, there is a set $S$ such that $\A'\cup\{S\}$ is also an antimatroid. Then $\A\cup\{X_j\cup S\}$ is itself an antimatroid, as desired. The accessibility of $\A\cup\{X_j\cup S\}$ follows from the accessibility of $\A'$. The closure under unions of $\A\cup\{X_j\cup S\}$ follows from the closure under unions of $\A'$, together with the fact that (because of how $j$ was chosen) every superset of $X_j\cup S\cup\{x\}$ belongs to $\A$.

If, on the other hand, $X_0\setminus\{x\}$ does not belong to $\A$, let $T$ be the union of the supersets of $X_j$ that belong to $\A$ and that do not contain $x$.  $T$ belongs to $\A$, and there exists an element $y$ of $\A$ with $y\notin T$; we choose $S=T\cup\{y\}$ as the set to add to $\A$. This addition preserves accessibility (the element $y$ may be removed from $S$ to produce another set in $\A$) and closure under unions (any union of $S$ with a set containing $x$ lies in the interval between $X_{i-1}$ and $X_0$ and therefore belongs to $\A$; any union of $S$ with a  set in $\A$ that does not contain $x$ results in $S$ itself as the result, because all sets in $\A$ that do not contain $x$ are subsets of $T$).
\end{proof}

Given a set of $n$ labeled elements, we may define a tree $T$ in which each node represents an antimatroid on these elements: $T$ is rooted at the power set of the elements, and each non-root node $\A$ has as its parent an antimatroid $A\cup\{S\}$ given by Lemma~\ref{lem:rev-search-parent}. We wrote a computer program to traverse this tree using \emph{reverse search}~\cite{AviFuk-DAM-96} (a variant of depth first search optimized for implicitly defined trees such as this one) and used it to list all antimatroids on five or fewer labeled elements. With additional heuristics for eliminating duplicate copies of isomorphic antimatroids we were able to extend the search to all antimatroids on six unlabeled elements. Our program determined that there are no counterexamples to the $\frac13$~--~$\frac23$ conjecture with six or fewer elements. Figure~\ref{fig:extremes} shows three antimatroids found by our program that have $\delta(\A)=\frac13$ but that do not represent partial orders; two of them have convex dimension three, showing that the conjecture that width-three partial orders have balance strictly better than~$\frac13$ does not generalize to antimatroids.

\subsection{Convex dimension two}

By using double ladders, we may generalize to antimatroids the theorem of Linial~\cite{Lin-SJoC-84} that partial orders of width two are balanced.

\begin{theorem}
\label{thm:convex-dim-2}
Let $\A$ be an antimatroid with convex dimension two. Then $\A$ contains a double ladder and hence is balanced.
\end{theorem}

\begin{proof}
An antimatroid $\A$ of convex dimension two may be defined by a pair of basic words $X$ and $Y$; the sets of the antimatroid all have the form of the union of a prefix of $X$ and a prefix of $Y$.
. We assume without loss of generality that $X$ and $Y$ begin with different elements $x_0$ and $y_0$, for if $x_0=y_0$, then all basic words begin with that element and we may obtain a smaller antimatroid by removing it from both words, after which the result follows by induction.

Let $X'=(x_0, x_1,\dots x_j)$ be the prefix of $X$ up to but not including the position at which $y_0$ occurs in $X$, and let $Y'=(y_0,y_1,\dots y_k)$ be the prefix of $Y$ up to but not including the position at which $x_0$ occurs in $Y$. Then $(x_0,Y')$ is a ladder: in any basic word, the elements of $Y$ occurring before $x_0$ must be in sequence order, and unless $x_0$ is first in a basic word, it can be swapped with its predecessor in $Y$ to form another basic word. Symmetrically, $(y_0,X')$ is a ladder, and $(X',Y')$ is a double ladder. The existence of a balanced pair follows from Theorem~\ref{thm:double-ladder-is-balanced}.
\end{proof}

\subsection{Height two}

We define the \emph{height} of an antimatroid to be the height of its path poset. An antimatroid has height at most two if and only if all of its elements are either initial or final. The basic words of a height-two antimatroid may be described by specifying, for each final element $x$, a monotonic Boolean function $f_x$ whose input is a vector of indicator variables for the initial elements that are true if an initial element belongs to a given prefix of a basic word and false otherwise, and whose output is a single Boolean value that is true when $x$ may be included next after the given prefix and false otherwise. 
The antimatroid defined from a partial order of height two itself has height two.
The height-two antimatroids that can be described by partial orderings are exactly the ones for which each of the functions  $f_x$ is the conjunction of some subset of its variables.

The following result generalizes the existence of balanced pairs in height-two partial orderings~\cite{TroGehFis-Order-92}, and we believe that its proof is significantly simpler. It follows from the fact that every height-two antimatroid with at least seven elements contains either a large number of initial elements or a large number of final elements, together with our computational results on antimatroids with at most six elements.

\begin{theorem}
\label{thm:height-2}
Let $\A$ be an antimatroid with height at most two. Then either $\A$ has one or two elements and defines a total order on its elements, or it contains a balanced pair.
\end{theorem}

\begin{proof}
Let $\A$ be an antimatroid with height at most two. If $\A$ has at most six elements, our computational searches show that it must be balanced. If $\A$ has at least seven elements, at least half of which are initial, it is balanced by Theorem~\ref{thm:many-initial}. And if it has at least seven elements but fewer than half of them are initial, then the remaining elements are final and comprise more than half of the elements, and again $\A$ is balanced by Theorem~\ref{thm:many-initial}.
\end{proof}

\section{Antimatroids in particular}

In this section we study several specific subclasses of matroids 
that we can prove always obey the $\frac13$--$\frac23$ conjecture. These matroids will be constructed from graphs, using one of the following two constructions.

The \emph{node search antimatroid} of a graph $G$ with a designated source node~$s$~\cite{Nak-DAM-03} has as its elements the vertices of $G$ other than $s$ itself. A subset~$S$ of vertices belongs to the antimatroid if $S\cup\{s\}$ forms the vertex set of a connected subgraph of~$G$. The basic words of a node search antimatroid give the possible orderings in which the vertices of~$G$ may be explored, starting from~$s$, subject to the rule that a vertex may be explored only after at least one of its neighbors has already been explored.

Another family of antimatroids may be defined from the \emph{chordal graphs}. An \emph{elimination ordering} of a graph $G$ is a total ordering of its vertices such that, for each vertex $v$, the neighbors of $v$ that occur later than $v$ in the ordering form a clique. Equivalently, for every three-vertex induced path in~$G$, one of the two path endpoints must be ordered earlier than the midpoint. A graph is chordal if and only if it has an elimination ordering, and the elimination orderings of a chordal graph form the basic words of an antimatroid~\cite{Chv-RTCO-09,Jam-CRCG-82,KorLovSch-Greedoids-91,Shi-DAM-84}.

\subsection{Elimination orderings of $k$-trees}

A \emph{$k$-tree} is a graph that can be formed from a $(k+1)$-vertex complete graph by adding zero or more vertices in some order such that, at the time each vertex is added, its neighbors form a $k$-vertex clique~\cite{Ros-DM-74}. The $k$-trees are chordal~\cite{Ros-DM-74}; their elimination orders are the reverses of the orders in which they can be built up by adding vertices.

\begin{theorem}
\label{thm:k-tree}
Let $\A$ be the antimatroid of elimination orderings of a $k$-tree for any $k\ge 1$. Then $\A$ contains a double ladder and hence is balanced.
\end{theorem}

\begin{figure}
\centering\includegraphics[scale=0.75]{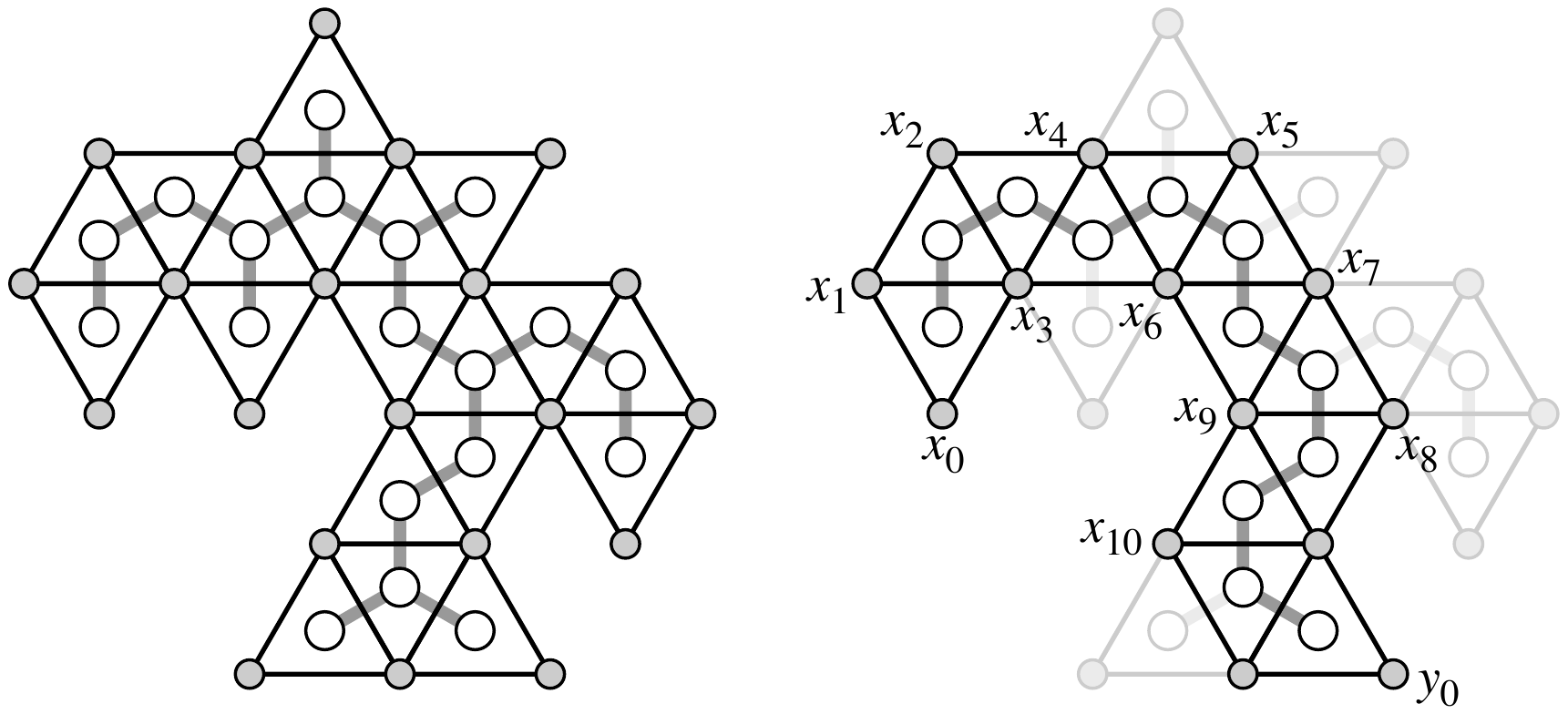}
\caption{Left: A 2-tree $G$ (small shaded vertices and thin black edges) and the tree $T$ of its maximal cliques (large white nodes and thick shaded edges). Right: the path between two leaves of $T$, and one of the two ladders defined for it in the proof of Theorem~\ref{thm:k-tree}.}
\label{fig:2-tree-ladder}
\end{figure}

\begin{proof}
Let $G$ be a $k$-tree with $n$ vertices. It is straightforward to prove by induction on $n$ that $G$ has $n-k$ maximal cliques, each of $k+1$ vertices, and that there exists a tree $T$ having these cliques as its nodes, with an edge between two nodes exactly when their intersection is a $k$-vertex clique (Figure~\ref{fig:2-tree-ladder}, left). The orderings in which $G$ may be built up by adding simplicial vertices to a $(k+1)$-vertex clique are in one-to-one correspondence with (the reverses of) the elimination orderings of $T$; however, these orderings are not quite in one-to-one correspondence with the elimination orderings of $G$ itself, because they do not specify the order of elimination of the final $k+1$ vertices. The simplicial vertices of $G$ are in one-to-one correspondence with the leaves of $T$ each simplicial vertex belongs to a single maximal clique in $G$, which must be a leaf vertex of $T$, and (unless $T$ contains a single node and $G$ is itself a clique) each leaf vertex of $T$ corresponds to a maximal clique of $G$ that contains exactly one simplicial vertex. In particular, $G$ has at least two simplicial vertices, which we call $x_0$ and $y_0$.

Let $\pi$ be the path in $T$ between the leaves corresponding to $x_0$ and $y_0$, let $\ell$ be the number of edges in $\pi$, and let $G'$ be the subgraph of $G$ induced by the vertices that belong to the cliques in $\pi$. Let $x_i$ (for $i=0$, $1$, $2$, \dots, $\ell-1$) be the order in which the vertices of $G'$ may be eliminated, starting with $x_0$, so that only the maximal clique containing $y_0$ remains, as shown in Figure~\ref{fig:2-tree-ladder}, right; define $y_i$ symmetrically. Then, we claim, these two sequences $X=(x_0,x_1,\dots)$ and $Y=(y_0,y_1,\dots)$ form a double ladder. For notational convenience, we let $k_0$ be the maximal clique containing $x_0$ in $\pi$, and for $i>0$ let $k_i$ be the maximal clique adjacent to $k_{i-1}$ in $\pi$. Then $k_i$ contains $x_i$ but does not contain $x_j$ for any $j<i$. For any two adjacent cliques $k_{i-1}$ and $k_i$, the clique $k_{i-1}$ differs from $k_i$ only in the inclusion of $x_{i-1}$ and the exclusion of a vertex different from $x_i$.

For each $0<i\le \ell-1$, the edge $x_{i-1}x_i$ belongs to an induced path with $x_{i-1}$ as one endpoint and $y_0$ as the other endpoint. In particular, if we let $f$ be the function that maps $x_i$ to the vertex that belongs to $k_i$ but not to $k_{i-1}$, then $x_{i-1}$--$x_i$--$f(x_i)$--$f(f(x_i))$--\dots--$y_0$ is an induced path. From this it follows that, in any basic word of the elimination antimatroid for $\A$, either $x_{i-1}$ or $y_0$ must occur prior to $x_i$, the first defining property for $(y_0,X)$ to be a ladder.

If $z$ is adjacent to a vertex $x_i$ but does not belong to $G'$, then $z$ must be nonadjacent to $x_\ell$ for some vertex $x_\ell\in k_i$ (otherwise, $k_i\cup\{z\}$ would form a $(k+2)$-vertex clique), and by the same argument as above, in any basic word of $\A$, either $z$ or $y_0$ must occur prior to $x_i$.
Therefore, if $x_i$ is the element of $X$ that occurs closest to and prior to $y_0$ in some basic word, any such $z$ must occur prior to $x_i$. Therefore, as long as $i<\ell-1$, swapping $x_i$ for $y_0$ leads to another basic word, because all the elements between $x_i$ and $y_0$ in the swap must be nonadjacent to $x_i$. In the special case that $i=\ell-1$, any of the $k$ common neighbors of $x_i$ and $y_0$ may appear between $x_i$ and $y_0$ in the basic word, but these vertices remain simplicial no matter which of $x_i$ or $y_0$ is eliminated first, so again the swap preserves the validity of the basic word. The ability to perform this type of swap is the second defining property for $(y_0,X)$ to be a ladder.

A symmetric argument shows that $(x_0,Y)$ is also a ladder, so $(X,Y)$ is a double ladder and $\A$ is balanced by Theorem~\ref{thm:double-ladder-is-balanced}.
\end{proof}

\subsection{Elimination orderings of block graphs}

\begin{figure}[t]
\centering\includegraphics[scale=0.6]{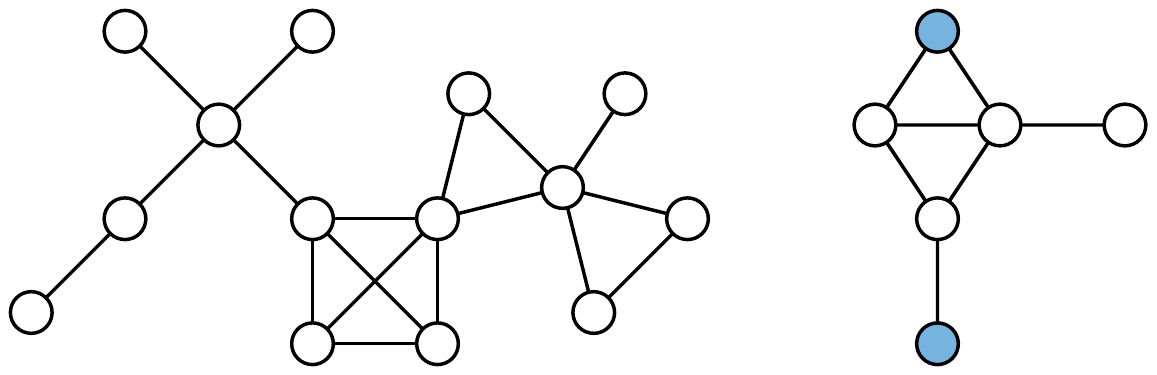}
\caption{Left: a block graph. Right: a ptolemaic graph whose antimatroid of elimination orderings has no double ladder, with a balanced pair of vertices marked.}
\label{fig:block-ptolemaic}
\end{figure}

The \emph{block graphs} (Figure~\ref{fig:block-ptolemaic}, left) are a subclass of the chordal graphs consisting of the graphs in which each block (biconnected component) is a clique~\cite{Har-CMB-63}. Equivalently, the block graphs are the chordal graphs in which each two maximal cliques intersect in at most a single vertex~\cite{How-JCTB-79}, making them in some sense as far as possible from the $k$-trees, in which many pairs of maximal $(k+1)$-cliques intersect in $k$-cliques.

\begin{theorem}
\label{thm:block-elim}
Let $\A$ be the antimatroid of elimination orderings of a block graph with more than one vertex. Then $\A$ contains a double ladder and hence is balanced.
\end{theorem}

\begin{proof}
By assumption, the given antimatroid $\A$ is defined from the set of elimination orderings of a block graph.
If the given block graph $G$ consists only of isolated vertices, then any two of these vertices form an indistinguishable pair and the result is immediate. Similarly, if any component of $G$ forms a nontrivial clique, any two of its vertices form an indistinguishable pair. Otherwise, let $X$ and $Y$ be two distinct blocks of $G$ that form leaves of the same tree of the block-cut forest of $G$ (that is, only one vertex of $X$ and only one vertex of $Y$ is a cut vertex), and let $x_0\in X$ and $y_0\in Y$ be any two vertices in these blocks that are not cut vertices. Let $X=(x_0,x_1,\dots)$ be the sequence of nodes on the shortest path from $x_0$ to $y_0$, up to but not including $y_0$, and let $Y=(y_0,y_1,\dots)$ be the sequence of nodes in the opposite order on the same path, up to but not including $x_0$.

Then the valid elimination orderings for $G$, restricted to the path from $x_0$ to $y_0$, must interleave nodes from the two ends of the path, from which it follows that the pair $(X,Y)$ obeys the first defining property of a double ladder, that the elements occur in the sequence order.

Suppose that $w$ is a basic word in which $x_0$ occurs later than $y_0$, and $y_i$ is the nearest path vertex to $x_0$ that occurs earlier than $x_0$ in $w$. Let $N$ be the vertices that are adjacent to $y_i$ but not adjacent to $y_{i+1}$ (not including $y_{i+1}$ itself).
Then in order for $y_i$ to have a clique as its set of later neighbors in the ordering, all vertices in $N$ must occur earlier than $y_i$ in the ordering. If some vertex $z\in N$ is the middle vertex of a three-vertex induced path for which $y_i$ is an endpoint, then either $z=y_{i+1}$ (and is eliminated after $x_0$ by assumption) or $y_i$ can be replaced by $y_{i+1}$ to produce a different three-vertex induced path. Thus, in either case, the elimination of $y_i$ has no effect on whether $z$ can be eliminated. It follows that swapping $x_0$ and $y_i$ produces another valid basic word, the second defining property of a double ladder.

We have identified a double ladder $(X,Y)$ in the given antimatroid $\A$, from which the existence of a balanced pair follows from Theorem~\ref{thm:double-ladder-is-balanced}.
\end{proof}

Every block graph is \emph{distance-hereditary}:  all induced paths with the same endpoints have the same length. However, Theorem~\ref{thm:block-elim} does not generalize to arbitrary chordal distance-hereditary graphs (also called \emph{ptolemaic graphs}): Figure~\ref{fig:block-ptolemaic} (right) shows a \emph{ptolemaic} or chordal distance-hereditary graph whose antimatroid of elimination orderings does not contain any double ladders, although its most balanced pair of elements (marked in the figure) has balance $50/102$.

\subsection{Node search orderings of distance-hereditary graphs}

Although the elimination antimatroids of distance-hereditary graphs may not have double ladders, their node search antimatroids do have them:

\begin{theorem}
\label{thm:distance-hereditary}
Let $\A$ be the node search antimatroid of a distance-hereditary graph with a given source vertex, and suppose that $\A$ has more than one basic word. Then $\A$ contains a double ladder and hence is balanced.
\end{theorem}

\begin{proof}
Let $G$ be a distance-hereditary graph with designated source vertex $s$, and let $\A$ be its node search antimatroid. Define the power of a vertex $v$ in $G$ to be the number of neighbors of $v$ that are farther from $s$ than $v$ is. If $G$ is a path and $s$ is an endpoint of the path then $\A$ has only one basic word; otherwise, some of the vertices have power at least two while others (e.g. the ones farthest from $s$) do not. Let $d$ be the largest distance from $s$ at which there exist vertices with power at least two, let $B$ be the bipartite graph of the edges in $G$ with one endpoint at distance $d$ and the other endpoint at distance $d+1$, and let $H$ be a connected component of $B$ that contains at least one vertex of power at least two. Let $P$ be the set of vertices in $H$ at distance $d$ from $s$ (not all of which necessarily have power two or more), and let $Q$ be the set of vertices in $H$ at distance $d+1$ from $s$ (necessarily all with power zero or one).

We observe that, if $u$ is at distance $d-1$ from $s$ and is adjacent to at least one vertex in $P$, then $u$ must be adjacent to all vertices in $P$. For, if there were vertices $v$ and $v'$ in $P$ such that $u$ were adjacent to $v$ but not $v'$, then there would be an induced path in $G$ from $v'$ to $v$ via $H$ (or possibly directly by an edge $vv'$) and then from $v$ to $s$ via $u$, of length greater than $d$, violating the distance-hereditary property of~$G$.

For each vertex $w\in Q$, let $N(w)$ denote the set of neighbors of $w$ that belong to $P$.
Then, for any two vertices $w$ and $w'$ in $Q$, the two sets $N(w)$ and $N(w')$ must either be disjoint or nested. If $w$ and $w'$ are adjacent then more strongly $N(w)=N(w')$, by an argument similar to the one above. Otherwise, if $N(w)$ and $N(w')$ could be neither disjoint nor nested then we could find a vertex $v\in N(w)\setminus N(w')$, a second vertex $v'\in N(w')\setminus N(w)$, and a third vertex $v''\in N(w)\cap N(w')$. But then, the induced path $w$--$v''$--$w'$ would have a different length from the induced path $w$--$v$--$u$--$v'$--$w'$ (if $v$ and $v'$ are nonadjacent and $u$ is a common neighbor at distance $d-1$ from $s$) or $w$--$v$--$v'$--$w'$ (if $v$ and $v'$ are adjacent), in either case violating the distance-hereditary property.

\begin{figure}[t]
\centering\includegraphics[scale=0.45]{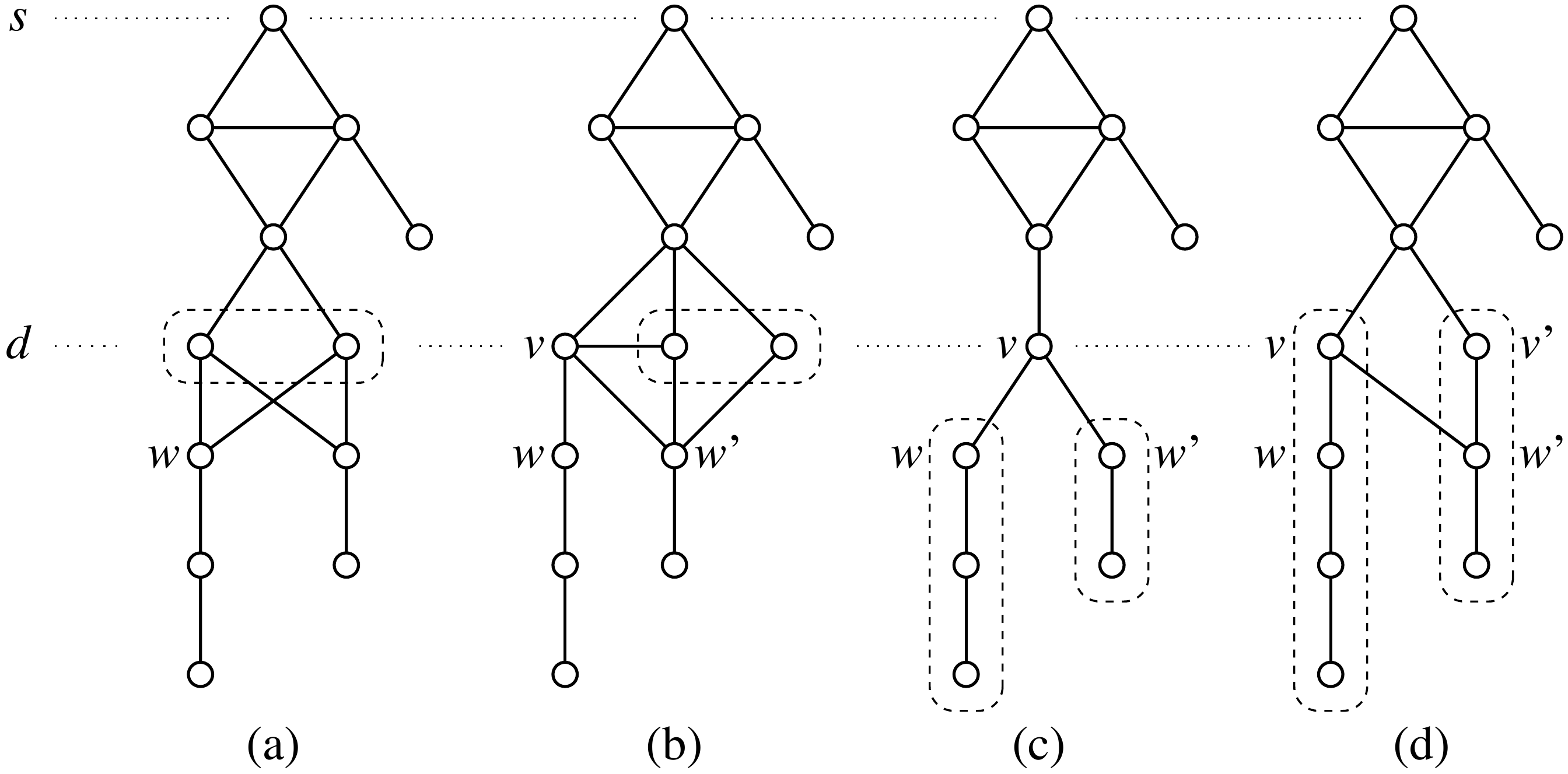}
\caption{Cases for the proof of Theorem~\ref{thm:distance-hereditary}. (a) If $|N(w)|\ge 2$, the vertices in $N(w)$ are indistinguishable. (b) If $|N(w')|\ge 3$, the vertices in $N(w')\setminus N(w)$ form indistinguishable elements of $\A$, even though they may not be isomorphic as graph vertices. (c) If $N(w)=N(w')=\{v\}$, the paths from $w$ and $w'$ form a double ladder. (d) If $N(w)=\{v\}$ and $N(w')=\{v,v'\}$, the paths from $v$ through $w$ and from $v'$ through $w'$ form a double ladder. In each case the indistinguishable vertices or double ladder paths are shown as surrounded by dashed curves.}
\label{fig:dh-cases}
\end{figure}

Choose $w$ in $Q$ such that $N(w)$ is minimal. If $|N(w)|>1$, then any two vertices in $N(w)$ are indistinguishable: for they have the same sets of neighbors with distances $d\pm 1$ to $s$, and the edges connecting pairs of vertices in $N(w)$ are irrelevant for the node searching antimatroid. If there are two indistinguishable elements, they form a trivial double ladder and we are done with the proof;  this case is depicted in Figure~\ref{fig:dh-cases}(a).
Otherwise, $|N(w)|=1$, and we let $v$ be the sole member of $N(w)$. Necessarily $v$ must have power two or more, for otherwise $v$ and $w$ would form by themselves a component of $B$ that does not contain any vertices of high power, contradicting the way we chose~$H$.

Now choose again $w'\in Q$ to be a neighbor of $v$ such that $N(w')$ is a minimal superset of $\{v\}$.  Because of our choice of $d$, $w$ and $w'$ as well as all nodes farther from $s$ have power zero or one; that is to say, the sets of nodes reachable from $s$ via induced paths through $w$ and through $w'$ (including $w$ and $w'$ themselves) form paths.
If $N(w')$ contains three or more vertices, then any two of them that are not $v$ are indistinguishable for the same reason as before, and we are done. Figure~\ref{fig:dh-cases}(b) illustrates this case. If $N(w)=N(w')$, as in Figure~\ref{fig:dh-cases}(c), then the two paths reachable via $w$ and $w'$ form a double ladder and we are done.

In the remaining case, shown in Figure~\ref{fig:dh-cases}(d), $N(w')=\{v,v'\}$ for some vertex $v'$.
Let $p$ be the path of vertices starting with $v$ and $w$ and continuing through the remaining vertices reachable via $w$. Let $p'$ be the path of vertices starting with $v'$ and $w'$ and continuing through the remaining vertices reachable via $w'$.
Then, in any basic word of $\A$, until both $v$ and $v'$ have been included, all elements from one of the two paths must occur in path order, the first defining property of a double ladder. Additionally, the only elements that are outside $p$ but that are reachable via a path through a vertex of $p$ are also reachable by a path through $v'$, and the only elements that are outside $p'$ but that are reachable via a path through a vertex of $p'$ are also reachable by a path through $v$. Therefore, swapping $v$ or $v'$ for the nearest earlier element of the other path always produces another basic word, the second defining property of a double ladder. Thus, $(p,p')$ forms a double ladder, completing the proof.
\end{proof}

\subsection{Split graphs}

A \emph{split graph}~\cite{FolHam-SE-77,TysChe-BSSR-79} is a graph whose vertices can be partitioned into a clique and an independent set. The vertices in the independent set form initial elements for the elimination ordering antimatroid; however, the vertices in the clique of a split graph do not necessarily form final elements, and we have not succeeded in proving that all elimination antimatroids of split graphs are balanced. The strongest result in this direction that we have found applies to distance-hereditary split graphs (equivalently ptolemaic split graphs or gem-free split graphs), a class of graphs studied by Branst\"adt et al.~\cite{BraDraLe-TCS-05} that includes the graph of Figure~\ref{fig:block-ptolemaic}(right).

\begin{theorem}
\label{thm:dh-split}
Let $\A$ be the antimatroid of elimination orderings of a distance-hereditary split graph $G$. Then  either $\A$ contains two indistinguishable elements, or at least half of the vertices of $G$ are simplicial. In either case, $\A$ is balanced.
\end{theorem}

\begin{proof}
Let $G$ be a distance-hereditary split graph, and $K$ and $I$ be the clique and independent set of a split decomposition of $G$; if $G$ may be split into a clique and an independent set in multiple ways, let $K$ be minimal among all possible choices of a clique in such a decomposition. Because $G$ is distance-hereditary, the neighborhoods of any two vertices in $I$ must be either disjoint or nested, for if two vertices $v$ and $v'$ in $I$ had neighborhoods that were neither disjoint nor nested then there would exist neighbors $w$, $w'$, and $w''$, with $w$ adjacent to $v$ but not $v'$, $w'$ adjacent to $v'$ but not $v$, and $w''$ adjacent to both. (The subgraph formed by these five vertices is called a \emph{gem}.) But then the induced paths $vww'v'$ and $vw''v'$ would have different lengths, contradicting the assumption that $G$ is distance-hereditary.

Now let $H$ be the bipartite graph that connects each vertex in $K$ to the adjacent vertices in $I$ whose neighborhoods are minimal. Every vertex $v$ in $K$ must have a neighbor in $H$, for otherwise we could move $v$ from $K$ to $I$ contradicting the assumed minimality of $K$. If a vertex $v\in K$ has two neighbors in $H$, those two neighbors must be indistinguishable. And if a vertex $v\in I$ has two neighbors in $K$, then those two neighbors must again be indistinguishable. The only remaining possibility is that $H$ is a perfect matching on $K$ and that $I$ contains at least as many vertices as $K$.

If $G$ contains two indistinguishable vertices then its antimatroid $\A$ of elimination orderings is clearly balanced, and if at least half of the vertices of $G$ belong to $I$ (and are therefore simplicial) then $\A$ is balanced by Theorem~\ref{thm:many-initial}.
\end{proof}

For node search antimatroids of arbitrary split graphs we can apply Theorem~\ref{thm:height-2}.

\begin{theorem}
\label{thm:node-search-height-2}
Let $\A$ be the node search antimatroid of a split graph in which the source node belongs to the clique of the split partition. Then $\A$ has height at most two and therefore is balanced.
\end{theorem}

\begin{proof}
Let $\A$ be the node search antimatroid of a split graph in which the source node belongs to the clique of the split partition. Then every vertex in this clique is adjacent to the source node, and hence forms an initial element of $\A$. Every other vertex is simplicial, and hence forms a final element of $\A$. Therefore, $\A$ has height two and is balanced by Theorem~\ref{thm:height-2}.
\end{proof}

A height-two antimatroid is the node search antimatroid of a split graph exactly when each of the monotonic functions $f_x$ defining it is a disjunction of some subset of its arguments.

\section{Conclusions}

We have identified several patterns that, when present in a set of orderings, ensure that it is balanced. Using these patterns, we have shown many specific classes of antimatroid to be balanced, and found antimatroid-theoretic generalizations of the theorems that partially ordered sets of width two or of height two are balanced. Based on these results, the evidence for the $\frac13$~--~$\frac23$ conjecture for antimatroids seems strong.

Antimatroids are not the most general structure from which one can define a set of orderings, and it would be of interest to explore the possibility of extending the $\frac13$~--~$\frac23$ conjecture to more general structures.  A natural candidate for a type of set system to explore, intermediate between antimatroids and arbitrary accessible set families, would be the full greedoids; however, we are not prepared to make a conjecture on whether they are balanced.

\subsection*{Acknowledgements}

This work was supported in part by NSF grants
0830403 and 1217322 and by the Office of Naval Research under grant
N00014-08-1-1015.

\raggedright
\bibliographystyle{abuser}
\bibliography{ab}
\end{document}